%% file: On_the_negative_K-theory_of_singular_varieties.tex
\documentclass{amsart}
\usepackage{setspace}
\usepackage[all,cmtip]{xy}
\usepackage{amsmath, amsthm, amssymb, amsfonts}
\usepackage{amsxtra, amscd, graphicx}
\usepackage{endnotes}
\usepackage[all,cmtip]{xypic}
\entrymodifiers={+!!<0pt,\fontdimen22\textfont2>}
\usepackage{mathrsfs}
\usepackage{verbatim}
\usepackage{mathdots}
\usepackage[pdftex,colorlinks, linkcolor=blue]{hyperref}

\input{mymacros}

\begin{document}


\title{On the negative $K$-theory of singular varieties}
\author{Justin Shih}
\address{}
\email{justin.shih@gmail.com}
\date{}
\keywords{Algebraic $K$-theory, Negative $K$-theory}
\subjclass[2010]{19D35, 19E08}
\begin{abstract}
Let $X$ be an $n$-dimensional variety over a field $k$ of characteristic zero, regular in codimension $1$ with singular locus $Z$. In this paper we study the negative $K$-theory of $X$, showing that when $Z$ is sufficiently nice, $K_{1-n}(X)$ is an extension of $KH_{1-n}(X)$ by a finite dimensional vector space, which we compute explicitly. We also show that $KH_{1-n}(X)$ almost has a geometric structure. Specifically, we give an explicit $1$-motive $[L \rightarrow G]$ and a map $G(k) \rightarrow KH_{1-n}(X)$ whose kernel and cokernel are finitely generated abelian groups.
\end{abstract}

\maketitle

\input{introduction}
\input{descent_for_KH}
\input{computation_continued}
\input{1-motive}
\input{independence}
\input{NK}

\section*{Acknowledgements}
The author is indebted to his advisor, Christian Haesemeyer, for his patience and mathematical guidance. The author would also like to thank Adam Massey and Wanshun Wong for insightful discussions.

\bibliographystyle{amsplain}
\bibliography {articles}
\end{document}

%% file: mymacros.tex
\theoremstyle{plain}
\newtheorem{thm}{Theorem}[section]
\newtheorem*{thm*}{Theorem}
\newtheorem{prop}[thm]{Proposition}
\newtheorem*{prop*}{Proposition}

\newtheorem*{conj*}{Conjecture}
\newtheorem{lem}[thm]{Lemma}
\newtheorem*{lem*}{Lemma}
\newtheorem{cor}[thm]{Corollary}
\newtheorem*{cor*}{Corollary}

\theoremstyle{definition}
\newtheorem*{ack*}{Acknowledgements}
\newtheorem*{app*}{Application}
\newtheorem*{apps*}{Applications}
\newtheorem{defn}[thm]{Definition}
\newtheorem*{defn*}{Definition}
\newtheorem{eg}[thm]{Example}
\newtheorem*{eg*}{Example}
\newtheorem*{egs*}{Examples}

\newtheorem*{ex*}{Exercise}

\newtheorem*{quest*}{Question}
\newtheorem{rem}[thm]{Remark}
\newtheorem*{rem*}{Remark}

\newtheorem*{rems*}{Remarks}
\newtheorem*{prob*}{Problem}
\newtheorem*{soln*}{Solution}

\providecommand{\Z}{{\mathbb{Z}}}
\providecommand{\Q}{\mathbb{Q}}
\providecommand{\C}{{\mathbb{C}}}
\providecommand{\cdh}{\mathrm{cdh}}
\providecommand{\Zar}{\mathrm{Zar}}

\providecommand{\Bl}{\operatorname{Bl}}
\providecommand{\coker}{\operatorname{coker}}
\providecommand{\codim}{\operatorname*{codim}}
\providecommand{\Div}{\operatorname{Div}}
\providecommand{\Gr}{\operatorname{Gr}}
\providecommand{\im}{\operatorname{im}}
\providecommand{\length}{\operatorname{length}}
\providecommand{\NS}{\operatorname{NS}}
\providecommand{\Pic}{\operatorname{Pic}}
\providecommand{\sPic}{\mathbf{Pic}}
\providecommand{\rank}{\operatorname{rank}}
\providecommand{\Sing}{\operatorname{Sing}}
\providecommand{\Spec}{\operatorname{Spec}}
\providecommand{\sign}{\operatorname{sign}}

\newcommand{\leftexp}[2]{{\vphantom{#2}}^{#1}{#2}}
\renewcommand{\tilde}[1]{\widetilde{#1}}

%% file: introduction.tex

\section{Introduction} \label{section:introduction}
Historically, the computation of the algebraic $K$-theory of schemes has been a difficult problem. Progress has steadily been made over the past few decades, and in this paper we focus on the \emph{negative} $K$-theory of varieties in characteristic $0$. These groups tend to be more accessible than those in positive degree. For example, it is well known that the negative $K$-theory of regular schemes vanish. For singular schemes, some progress has been made for schemes in low dimension \--- for example, see Weibel \cite{Wei01} for the case of normal surfaces.

\medskip
Let $X$ be integral $n$-dimensional scheme ($n \ge 3$) of finite type over an algebraically closed field $k$ of characteristic zero, such that $X$ has only isolated singularities. In this paper, we give a full description of $K_{-2}(X)$ when $n = 3$, and partially generalize our findings to a description of $K_{1-n}(X)$ when $Z = \Sing(X)$ is either smooth or of codimension greater than $2$.

\medskip
With a little work, we establish an exact sequence

\begin{equation} \label{eq:exact sequence calculating K_n}
\xymatrix{
NK_{1-n}(X) \ar[r] & K_{1-n}(X) \ar[r] & KH_{1-n}(X) \ar[r] & 0
}
\end{equation}

\medskip
\noindent which computes $K_{1-n}(X)$. We compute the contributions $NK_{1-n}(X)$ and $KH_{1-n}(X)$ separately, and then determine how they fit together. More specifically, we compute $KH_{1-n}(X)$ and determine the image of $NK_{1-n}(X)$ in $K_{1-n}(X)$.

\medskip
When $X$ is a irreducible $n$-dimensional scheme of finite type over a field $k$ (not necessarily algebraically closed) of characteristic zero, with at most isolated singularities, we show that the image of $NK_{1-n}(X)$ in $K_{1-n}(X)$ is a finite-dimensional $k$-vector space.

\medskip
On the other hand, when $X$ is an integral $n$-dimensional scheme of finite type over an algebraically closed field $k$ of characteristic zero, such that $Z = \Sing(X)$ is either smooth over $k$ or of codimension greater than $2$, we relate $KH_{1-n}(X)$ to the largest torsion-free mixed Hodge structure $H$ in $H^n(X,\Z)$ of type $\{(0,0),(0,1),(1,0),(1,1)\}$ such that $\Gr_1^W \!\! H$ is polarizable.

\medskip
The main theorem, in the case that $X$ is a complex threefold with isolated singularities, exemplifies almost all of the interesting phenomena that occur in the general case.

\begin{thm}[Main theorem for $K_{-2}(X)$ of a complex threefold with isolated singularities]
Let $X$ be a integral three-dimensional variety of finite type over $\C$ with only isolated singularities. Then there is a short exact sequence

\begin{equation}
\xymatrix{
0 \ar[r] & V \ar[r] & K_{-2}(X) \ar[r] & KH_{-2}(X) \ar[r] & 0,
}
\end{equation}

\medskip
\noindent where $V$ is a finite-dimensional $\C$-vector space (explicitly computed in section \ref{sec:NK}), and $KH_{-2}(X)$ has the following description. Let $H \subseteq H^3(X,\Z)$ be the largest torsion-free mixed Hodge structure of type $\{(0,0),(0,1),(1,0),(1,1)\}$ such that $\Gr_1^W\!\!H$ is polarizable, and let $M = [L \longrightarrow G]$ be a $1$-motive corresponding to $H$ under Deligne's equivalence between $1$-motives and mixed Hodge structures of the given type. Then there is a left-exact sequence

\begin{equation}
\xymatrix{
0 \ar[r] & L(\C) \ar[r] & G(\C) \ar[r]^-{\alpha} & KH_{-2}(X)
}
\end{equation}

\medskip
\noindent such that $\coker(\alpha)$ is a finitely generated abelian group.
\end{thm}

\medskip
Morally, the main theorem says that $K_{-2}(X)$ is an extension of something that, up to some finitely generated abelian groups, is isomorphic to the $\C$-points of a $1$-motive, by a finite-dimensional vector space.

\subsection{Notation and outline of paper} \label{subsection:notation}
In this section, we introduce the problem and give a brief history. We then state the main result, establish notation, and give an outline of the paper. The computation of $KH_{1-n}(X)$ will take up the majority of this paper, and sections \ref{sec:descent for KH} through \ref{sec:independence of resolution} are dedicated to this computation. In section \ref{sec:descent for KH}, we take a (good) resolution of singularities for $X$, and then apply a descent argument to establish an exact sequence \eqref{eq:kh es, arbitrary n} computing $KH_{1-n}(X)$. We then compute each of the contributions in section \ref{sec:H^(n-1)(E,G_m)}. The main piece of $K_{1-n}(X)$ is a $1$-motive that arises out of this computation, which is the focus of section \ref{sec:1-motive}. These computations are done after picking a resolution of singularities for $X$, and it is natural to ask what parts, if any, are independent of the choice of resolution. We address this question in section \ref{sec:independence of resolution}. In section \ref{sec:NK}, we compute $NK_{1-n}(X)$ and describe its image in $K_{1-n}(X)$ under the map given in \eqref{}, wrapping up the computation of $K_{1-n}(X)$.

\medskip Throughout this paper, $k$ will denote a field of characteristic $0$. $Sch/k$ will denote the category of separated schemes over $k$ of finite type, and $Z = \Sing(X)$ will denote the singular locus of $X$. Starting in section \ref{sec:1-motive}, we will need to refer to both Picard groups and the schemes that represent them; to avoid confusion, we will let $\sPic(X)$ denote the Picard scheme of $X$, whenever it exists, and similarly for $\sPic^0(X)$.

%% file: descent_for_KH.tex

\section{Calculation of $KH_{1-n}(X)$}
\label{sec:descent for KH}
In this section, in addition to our assumptions about $X$ laid out in \ref{subsection:notation}, we will also assume that $X$ reduced, and either $\codim Z > 2$ or that $Z$ is smooth. We begin by considering a good resolution of singularities $p \colon \tilde{X} \longrightarrow X$, i.e. a proper birational map which is an isomorphism outside of $Z$, and such that the exceptional divisor $E$ is a simple normal crossing divisor. $KH$ satisfies \emph{cdh}-descent \cite{Hae04}, and applying it to our good resolution yields a long exact sequence of $KH$ groups

\begin{equation} \label{eq:cdh-descent kh les}
\xymatrix{
\cdots \ar[r] & KH_{-q}(X) \ar[r] & KH_{-q}(Z) \oplus KH_{-q}(\tilde{X}) \ar[r] & KH_{-q}(E) \ar[r] & \cdots
}
\end{equation}

\medskip
Since $\tilde{X}$ is smooth, its $KH$-groups agree with its $K$-groups, which vanish in negative degree (and the same if $Z$ is smooth). If $\codim Z > 2$, then the $K$-dimension theorem asserts that $K_{-q}(Z) = KH_{-q}(Z) = 0$ for all $q > n-2$. In either case, we obtain natural isomorphisms $K_{1-q}(E) \cong K_{-q}(X)$ for all $q > n-2$.

\medskip
When $Z = \coprod_i Z_i$ has more than one connected component, we will have $KH_{1-n}(X) \cong \oplus_i KH_{2-n}(E_i)$, where $E_i$ is the total transform of $Z_i$. We can compute each of these groups separately, so we may assume that $Z$ is connected. By Zariski's main theorem, $E$ will also be connected, so we will also assume that $E$ has $r$ irreducible components. In any case, to compute $K_{1-n}(X)$, we will compute $K_{2-n}(E)$ instead, using the fact that it has simple normal crossings.

\medskip
There is a \emph{cdh}-descent spectral sequence for $KH$, which we apply to $E$: \cite{Hae04}

\begin{equation} \label{eq:cdh-descent spectral sequence for KH}
E_2^{pq} = H^p_{\cdh}(E,aKH_{-q}) \Longrightarrow KH_{-p-q}(E),
\end{equation}

\medskip
\noindent where $a$ denotes sheafification in the \emph{cdh}-topology. The first thing to note is that since schemes are locally smooth in the \emph{cdh}-topology, the natural map $K \longrightarrow KH$ induces an isomorphism of sheaves $aK_q \cong aKH_q$. We now have a short lemma.

\begin{lem} \label{}
For $q \le 1$, we have the following isomorphism of sheaves on $Sch/k$:

\begin{spacing}{1}
\begin{equation}
a_{\cdh}K_q =
    \begin{cases}
        a_{\cdh}\mathbb{G}_m,   & q = 1 \\
        a_{\cdh}\Z,             & q = 0 \\
        0,                      & q < 0.
    \end{cases}
\label{}
\end{equation}
\end{spacing}
\end{lem}

\begin{proof}
First, since every \emph{cdh}-cover has a refinement by smooth schemes, and $K_q(U) = 0$ whenever $U$ is regular and $q < 0$, $a_{\cdh}K_q = 0$ when $q < 0$.

\medskip
We can sheafify both $K_0$ and $\Z$ in two steps, as follows:

\begin{equation} \label{}
\vcenter{\xymatrix{
K_0 \ar[r] \ar[d]^-{\rank} & a_{\Zar}K_0 \ar[r] \ar[d] & a_{\cdh}a_{\Zar}K_0 = a_{\cdh}K_0 \ar[d] \\
\Z \ar[r] & a_{\Zar}\Z \ar[r] & a_{\cdh}a_{\Zar}\Z = a_{\cdh}\Z
}}
\end{equation}

\medskip
The rank map $a_{\Zar}K_0 \longrightarrow \Z$ is an isomorphism on local rings, so $a_{\cdh}K_0 \longrightarrow a_{\cdh}\Z$ is an isomorphism. We have a similar argument for the sheaf $a_{\cdh}K_1$:

\begin{equation} \label{}
\vcenter{\xymatrix{
K_1 \ar[r] \ar[d] & a_{\Zar}K_1 \ar[r] \ar[d] & a_{\cdh}a_{\Zar}K_1 = a_{\cdh}K_1 \ar[d] \\
\mathbb{G}_m \ar[r] & \mathbb{G}_m \ar[r] & a_{\cdh}\mathbb{G}_m
}}
\end{equation}

\medskip
The map $a_{\Zar}K_1 \longrightarrow \mathbb{G}_m$ is an isomorphism on the stalks, since $K_1(R) = R^{\times}$ when $R$ is local.
\end{proof}

\medskip
When there is no ambiguity (for example, when we take cohomology groups), we will sometimes write $K_q$ for $a_{\cdh}K_q$, and similarly for $\mathbb{G}_m$ and $\Z$.

\medskip
The above lemma allows us to conclude that the descent spectral sequence resides in the fourth quadrant. Moreover, the spectral sequence degenerates at page $E_n$, because the \emph{cdh}-cohomological dimension of $E$ is at most $\dim E = n-1$ \cite{SV00}. In particular, the spectral sequence is bounded. The upshot of this analysis is that for $p+q = n-1$, the descent spectral sequence gives us an exact sequence

\begin{equation} \label{eq:kh es, arbitrary n}
\xymatrix{
H^{n-3}_{\cdh}(E,\Z) \ar[r]^-{d_2^{n-3,0}} & H^{n-1}_{\cdh}(E,\mathbb{G}_m) \ar[r] & KH_{2-n}(E) \ar[r] & H^{n-2}_{\cdh}(E,\Z) \ar[r] & 0.
}
\end{equation}

\medskip
To calculate $KH_{2-n}(E)$, we need to know about the map $d_2^{n-3,0} \colon H^{n-3}_{\cdh}(E,\Z) \longrightarrow H^{n-1}_{\cdh}(E,\mathbb{G}_m)$. We mentioned in the introduction that the case $n = 3$ differs from the general case $n > 3$; we see an example of this below.

\begin{lem}
In the case $n = 3$, the differential $d_2^{0,0}$ is zero.
\label{d_2^{0,0} is the zero map}
\end{lem}

\begin{proof}
Let $P$ be a closed point of $E$, and consider the diagram

\begin{equation}
\vcenter{\xymatrix{
K_0(E) \ar[r] \ar@{>>}[d]_-{\rank} & KH_0(E) \ar@{>>}[r] \ar[d] & E_{\infty}^{0,0}(E) \ar@{>->}[r] \ar[d] & E_2^{0,0}(E) \ar[d] \\
K_0(P) \ar@{=}[r]         & KH_0(P) \ar@{=}[r]         & E_{\infty}^{0,0}(P) \ar@{=}[r]          & E_2^{0,0}(P) = \Z
}}
\label{}
\end{equation}

\noindent obtained from naturality of both the map $K \longrightarrow KH$ and the descent spectral sequence. Since we assume $E$ is connected, the vertical map on the right, $H_{\cdh}^0(E,\Z) \longrightarrow H_{\cdh}^0(P,\Z)$ is an isomorphism. Furthermore, the rank map on the left is surjective since there are vector bundles on $E$ of any rank. A diagram chase shows that the map $E_{\infty}^{0,0}(E) \longrightarrow E_2^{0,0}(E)$ is an isomorphism.
\end{proof}

\medskip
Thus when $n = 3$, the descent spectral sequence reduces to a short exact sequence

\begin{equation} \label{eq: kh_ses, n = 3}
\xymatrix{
0 \ar[r] & H_{\cdh}^2(E,\mathbb{G}_m) \ar[r] & KH_{-1}(E) \ar[r] & H_{\cdh}^1(E,\Z) \ar[r] & 0. \\
}
\end{equation}

\medskip
To compute the \emph{cdh}-cohomology groups appearing in \eqref{eq:kh es, arbitrary n}, we take a small detour to recall several constructions associated to the simple normal crossing divisor $E$. The simplicial and semisimplicial schemes associated to $E$, denoted $\Delta_{\bullet}E$ and $\Delta_{\bullet}^{alt}E$ respectively, are constructed as follows.

\begin{equation} \label{eq: defn of semisimplicial, simplicial schemes associated to SNCD}
\begin{split}
\Delta_{\bullet} E & = \coprod_{i_0, \ldots, i_p} (E_{i_0} \times_E \cdots \times_E E_{i_p}) \\
\Delta_{\bullet}^{alt}E & = \coprod_{i_0 < \cdots < i_p} (E_{i_0} \times_E \cdots \times_E E_{i_p}),
\end{split}
\end{equation}

\medskip
\noindent with the face maps $d_{p,j}$ given by the natural projections from the fiber products and the degeneracy maps $s_{p,j}$ induced by the diagonal maps (isomorphisms) $E_j \longrightarrow E_j \times_E E_j$. The simplicial scheme $\Delta_{\bullet}E$ has both face and degeneracy maps, whereas the semisimplicial scheme $\Delta_{\bullet}^{alt}E$ has only face maps.

\medskip
Let $\mathcal{V}_k$ denote the additive category of $k$-varieties, where the objects are $k$-varieties, and the morphisms are formal $\Z$-linear combinations of actual morphisms of varieties. From $\Delta_{\bullet}^{alt}E$ we can construct a complex of varieties $C_{\bullet}(\Delta_{\bullet}^{alt}E)$ in $\mathcal{V}_k$ in the standard way by taking the differentials to be the alternating sum of the face maps, i.e. $\partial_p = \sum_i (-1)^id_i$ (and similarly for $\Delta_{\bullet}E$).

\medskip
A related construction is that of the \emph{dual complex} associated to $E$ which we denote, following \cite{Kol12}, by $\mathcal{D}(E)$. It is a CW-complex constructed as follows. For each component $E_i$ of $E$ we have a vertex, which we label $i$. Then for each (connected) component of each intersection $E_i \cap E_j$, we glue in a $1$-cell between vertices $i$ and $j$ \--- this is the $1$-skeleton of $\mathcal{D}(E)$. Proceeding inductively, we attach an $m$-cell onto the $(m-1)$-skeleton for each connected component of each $m$-fold intersection. Since $E$ has finitely many components, we will eventually stop gluing, and will be left with the CW-complex $\mathcal{D}(E)$.

\medskip
Additionally, different resolutions of $X$ yield different dual complexes $\mathcal{D}(E)$, but it turns out that the homotopy type of $\mathcal{D}(E)$ is independent of the choice of good resolution \cite[Theorem 1.2]{Ste06}. This fact is reflected in our use of the notation $\mathcal{DR}(X)$ to denote the homotopy type of $\mathcal{D}(E)$. On the other hand $\mathcal{D}(E)$ is in general only a cell complex and need not be a simplicial complex.

\medskip
\noindent It is convenient for simplifying upcoming calculations to investigate for which $X$ we can find a resolution whose exceptional divisor $E$ has $\mathcal{D}(E)$ a simplicial complex. Luckily, the answer turns out to be the best possible: such a resolution always exists. We begin with establishing an obvious criterion on $E$ to have $\mathcal{D}(E)$ be a simplicial complex. An intersection $\cap_{i \in I} E_i$ is by definition smooth in a simple normal crossing divisor, so it is the disjoint union of its components. We will call $\cap_{i \in I} E_i$ a \emph{bad intersection} if it is not irreducible. It turns out that bad intersections are the only obstruction for $\mathcal{D}(E)$ to be a simplicial complex.

\begin{lem}
Given a good resolution $p \colon \tilde{X} \longrightarrow X$, the dual complex $\mathcal{D}(E)$ is a simplicial complex if and only if each of the intersections $\bigcap_{i \in I} E_i$ is irreducible.
\label{simplicial dual complex criterion}
\end{lem}

\begin{proof}
Suppose $E$ has $m$ irreducible components, and that $\cap_{i \in I} E_i$ is a bad intersection. Then in the construction of the dual complex, we will have multiple $|I|$-cells glued in the same place, so $\mathcal{D}(E)$ cannot be simplicial.

\medskip
\noindent Conversely, suppose $\cap_{i \in I} E_i$ is irreducible, and consider the corresponding $|I|$-simplex $D_I$ in $\mathcal{D}(E)$. All of the faces of $D_I$ are in $\mathcal{D}(E)$, because any such face corresponds to a smaller intersection of the $E_i$, which must be nonempty. Furthermore, for any other subset $J$ of $\{1,\ldots,m\}$, we have $D_I \cap D_J = D_{I \cup J}$, which is a face of both.
\end{proof}

\medskip
\noindent The preceding criterion eliminates disconnected intersections (which correspond to multiple cells glued to the same vertices). We now proceed with the proof that a resolution of $X$ always exists with dual complex $\mathcal{D}(E)$ a simplicial complex. The idea behind the proof was communicated to us by J{\'a}nos Koll{\'a}r.

\begin{prop}
There exists a resolution of $X$ with exceptional divisor $E$ for which $\mathcal{D}(E)$ is a simplicial complex. Moreover, such a resolution can be obtained from any good resolution by further blowups.
\end{prop}

\begin{proof}
Let $p \colon \tilde{X} \longrightarrow X$ be a good resolution with exceptional divisor $E$. We will iteratively blow up enough closed subschemes so that the conditions of Lemma \ref{simplicial dual complex criterion} are satisfied. We begin blowing up components of bad intersections of the smallest dimension, then move up in dimension.

\medskip
\noindent Write $E = \cup_{i = 1}^m E_i$ as the union of its irreducible components, and suppose $E$ has no bad intersections of codimension $> r$. We will blow up components of bad intersections of codimension $r$ one by one; we claim that when we have blown them all up, the resulting divisor will not have any bad intersections of codimension $r$. Write $E_I = \cap_{i \in I} E_i$, and let $B_r(p)$ be the number of connected components $E_I^{(j)}$ that belong to some bad intersection $E_I$ (i.e. $E_I$ has more than one connected component). Fix a bad intersection $E_{I_0}$ of codimension $r$, and without loss of generality, blow up $\tilde{X}$ along the smooth irreducible center $Z = E_{I_0}^{(1)}$. We claim that if $p' \colon \Bl_{E_{I_0}^{(1)}} \tilde{X} \longrightarrow X$, then $B_r(p') = B_r(p)-1$. It is easy to verify that if we continue in this manner, we will eventually remove all of the bad codimension $r$ intersections of $E$. This finishes the proposition.
\end{proof}

\begin{defn}
We call a strong resolution $p \colon \tilde{X} \longrightarrow X$ an \emph{excellent} resolution if the exceptional divisor $E$ is a simple normal crossing divisor and $\mathcal{D}(E)$ is a simplicial complex.
\end{defn}

\medskip
One way that the above constructions are relevant to us is that the collection of maps $\{E_i \longrightarrow E\}_i$ is a \emph{cdh}-cover, so we get a \v{C}ech-to-derived spectral sequence

\begin{equation} \label{}
E_1^{p,q} = H^q_{\cdh}(\Delta_pE,K_m) \Longrightarrow H^{p+q}_{\cdh}(E,K_m),
\end{equation}

\medskip
\noindent for each $m$. We may replace the rows $E_1^{\bullet,q}$ in the $E_1$ page of this spectral sequence with the quasi-isomorphic complexes consisting of only the non-degenerate parts \cite[Lem. 8.3.7]{WHo94}; that is, we may replace $H^q_{\cdh}(\Delta_{\bullet}E)$ with $H^q_{\cdh}(\Delta_{\bullet}^{alt}E)$. By a result of Voevodsky \cite{VSF00}, we may also replace the \emph{cdh}-cohomology groups with Zariski cohomology groups, since the sheaves $a_{\cdh}K_m$ are homotopy invariant sheaves with transfers \cite[Sec. 3.4]{Voe00}. So we obtain

\begin{equation} \label{eq:cdh semisimplicial spectral sequence applied to K_m}
E_1^{p,q} = H^q_{\Zar}(\Delta_p^{alt}E,K_m) \Longrightarrow H^{p+q}_{\cdh}(E,K_m).
\end{equation}

\medskip
For this first quadrant spectral sequence, many terms are zero. First, $\Delta_p^{alt}E = \varnothing$ for $p > n-1 \ge \dim E$, so $E_1^{p,q} = 0$ for $p > n-1$. Additionally, since $\dim \Delta_p^{alt}E \le \dim E - p = n-1-p$, we have $E_1^{p,q} = 0$ for $p + q > n-1$.

\medskip
\noindent We first use this spectral sequence to compute the groups $H_{\cdh}^i(E,\Z)$.

\begin{lem} \label{H^i(E,Z) finitely generated}
$H_{\cdh}^i(E,\Z) \cong H^i(\mathcal{D}(E),\Z)$. In particular, these groups are finitely generated.
\end{lem}

\begin{proof}
In addition to the observations we have already made about the spectral sequence, we also have $E_1^{p,q} = 0$ for $q > 0$ since $\Z$ is flasque as a Zariski sheaf. So $H_{\cdh}^i(E,\Z)$ is just the cohomology of the complex

\begin{equation} \label{eq: complex yielding cohomology groups of D(E)}
\xymatrix{
 \cdots \ar[r] & H_{\Zar}^0(\Delta_{i-1}^{alt} E,\Z) \ar[r] & H_{\Zar}^0(\Delta_i^{alt} E,\Z) \ar[r] & H_{\Zar}^0(\Delta_{i+1}^{alt} E,\Z) \ar[r] & \cdots
}
\end{equation}

\medskip
\noindent in degree $i$. Since $H^0_{\Zar}(Y,\Z) = \Z$ for $Y$ smooth and connected, this complex is isomorphic to the cellular chain complex of $\mathcal{D}(E)$

\medskip
Furthermore, since the homotopy type of $\mathcal{D}(E)$ is independent of the choice of resolution \cite{Ste06}, we also have that $H_{\cdh}^i(E,\Z) \cong H^i(\mathcal{DR}(X),\Z)$ is also independent of the choice of resolution.
\end{proof}

\begin{eg}
This same approach allows us to calculate $KH_{-n}(X)$. Applying \emph{cdh}-descent for $KH$ to our resolution of singularities of $X$ yields $KH_{1-n}(E) \cong KH_{-n}(X)$; application of the descent spectral sequence then yields $KH_{1-n}(E) \cong E_2^{n-1,0} = H_{\cdh}^{n-1}(E,\Z)$. The above lemma then tells us that $KH_{1-n}(E) \cong H^i(\mathcal{D}(E),\Z) = H^i(\mathcal{DR}(X),\Z)$. $\hfill \square$
\end{eg}

\medskip
Applying Lemma \ref{H^i(E,Z) finitely generated} to equation \eqref{eq:kh es, arbitrary n}, we see that the kernel and cokernel of the map $H^{n-1}_{\cdh}(E,\mathbb{G}_m) \longrightarrow KH_{1-n}(E)$ are finitely generated. We may rephrase this result by saying that the cohomology group $H^{n-1}_{\cdh}(E,\mathbb{G}_m)$, which is generally large as we will see, approximates $KH_{1-n}(E)$, up to some finitely generated groups.

\medskip
We would now like to compute $H_{\cdh}^{n-1}(E,\mathbb{G}_m)$. We begin by analyzing the spectral sequence \eqref{eq:cdh semisimplicial spectral sequence applied to K_m}.

%% file: computation_continued.tex
\section{Simplifying the simplicial spectral sequence}
\label{sec:H^(n-1)(E,G_m)}

We begin with a small, well-known fact.

\begin{lem} \label{lem:G_m no higher cohomologies}
Let $Y$ be a smooth scheme over $k$. Then $H_{\Zar}^q(Y,\mathbb{G}_m) = 0$ whenever $q > 1$.
\end{lem}

\begin{proof}
The claim follows immediately from the explicit flasque resolution

\begin{equation} \label{eq:flasque resolution of G_m}
\xymatrix{
0 \ar[r] & \mathbb{G}_m \ar[r] & \mathscr{K}^{\times} \ar[r] & \mathrm{CaDiv} \ar[r] & 0,
}
\end{equation}

\medskip
\noindent where $\mathscr{K}$ is the sheaf of total quotient rings on $Y$, and $\mathrm{CaDiv}$ is the sheaf of Cartier divisors on $Y$.
\end{proof}

\medskip
\noindent Since $\Delta_{\bullet}^{alt}E$ is a smooth semisimplicial scheme, Lemma \ref{lem:G_m no higher cohomologies} tells us that the $E_1$ page of the spectral sequence \eqref{eq:cdh semisimplicial spectral sequence applied to K_m} with $m = 1$ only has two rows, and we need only compute $E_3^{n-1,0}$ and $E_3^{n-2,1}$. The $E_1$ page of this spectral sequence looks like

\begin{equation} \label{eq:simplicial spectral sequence E_1 page}
\vcenter{\xymatrix{
\cdots \ar[r] & \Pic(\Delta_{n-3}^{alt}E) \ar[r]^-{d_1^{n-3,1}} \ar@{-->}[drr]^-{d_2^{n-3,1}} & \Pic(\Delta_{n-2}^{alt}E) \ar[r] & 0 \ar[r]         & 0 \\
\cdots \ar[r] & k(\Delta_{n-3}^{alt}E)^{\times} \ar[r] & k(\Delta_{n-2}^{alt}E)^{\times} \ar[r] & k(\Delta_{n-1}^{alt}E)^{\times} \ar[r] & 0 \\
}}
\end{equation}

\medskip
In order to compute $E_{\infty}^{n-1,0} = E_3^{n-1,0}$, we need to determine the possibly nonzero differential $d_2^{n-3,1}$, which we have denoted using a dashed arrow in the diagram above. Applying the global sections of the resolution \eqref{eq:flasque resolution of G_m} for each $\Delta_p^{alt}E$ in each column yields the following diagram.

\begin{equation} \label{simplicial E0 page}
\vcenter{\xymatrix{
\cdots \ar@{-->}[r] & \Div(\Delta_{n-3}^{alt} E) \ar@{-->}[r] & \Div(\Delta_{n-2}^{alt} E) \ar@{-->}[r] & 0 \\
\cdots \ar@{-->}[r] & k(\Delta_{n-3}^{alt} E)^{\times} \ar@{-->}[r] \ar[u] & k(\Delta_{n-2}^{alt} E)^{\times} \ar@{-->}[r] \ar[u] & k(\Delta_{n-1}^{alt} E)^{\times} \ar[u]
}}
\end{equation}

\medskip
The face maps $d_i$ induce pullback maps on Picard groups, and the $E_1$ differentials are the alternating sum of these pullback maps. The dashed horizontal maps in the above diagram are the alternating sum of the pullbacks on the divisors themselves. They are dashed because they are not necessarily defined on all of the source, only on those divisors $\Div(\Delta_m^{alt}E)$ which intersect $\Div(\Delta_{m+1}^{alt}E)$ transversally. To remedy this, we will find a quasi-isomorphic subcomplex for which the horizontal maps are defined, then use this subcomplex to show that the map $d_2^{n-3,1}$ in \eqref{eq:simplicial spectral sequence E_1 page} is the zero map. Our current discussion motivates the following definition.

\begin{defn} \label{defn: group of good divisors}
For each $p$, we define the \emph{group of good divisors}

\begin{equation}
\Div_g(\Delta_p^{alt} E) = \{ D \in \Div(\Delta_p^{alt}E) \,\, | \,\, D \mbox{ intersects } \Delta_m^{alt}E \mbox{ transversally for all } m > p\}.
\label{}
\end{equation}

\end{defn}

\medskip
\begin{rem} \label{rem:}
By Bertini's theorem, this definition is equivalent to the one that instead requires the image of $D$ under any composition of any of the face maps $d_j$ to be defined. In addition, while the notation $\Div_g$ comes from Carlson \cite{Car85}, our definitions are slightly different. Carlson only requires that the image of $d_j$ is contained in $\Div(\Delta_{p+1}^{alt}E)$, instead of requiring that any composable composition of the $d_j$ is defined. Furthermore, Carlson's definition applies to a more general class of semisimplicial schemes, as we only define $\Div_g$ for semisimplicial schemes associated to the special class of simple normal crossing schemes.
\end{rem}

\medskip
\noindent We will now prove the following:

\begin{lem}[Moving Lemma]
For each $p$, let $A_p$ be the pullback

\begin{equation} \label{}
\vcenter{\xymatrix{
A_p \ar[r] \ar[d] & \Div_g(\Delta_p^{alt} E) \ar[d] \\
k(\Delta_p^{alt} E)^{\times} \ar[r]^-{\beta_p} & \Div(\Delta_p^{alt} E)
}}
\end{equation}

\noindent where $\beta_p$ is the rational function-to-divisor map. Then the vertical maps are a quasi-isomorphism of complexes.
\end{lem}

\medskip
\begin{proof}
We add in the horizontal kernels and cokernels to the diagram above, and label the vertical maps:

\begin{equation}
\vcenter{\xymatrix{
\ker(\alpha_p) \ar[r] \ar@{=}[d] & A_p \ar[r]^-{\alpha_p} \ar[d]^-{v_k} & \Div_g(\Delta_p^{alt} E) \ar[r] \ar[d]^-{v_{\Div}} & \coker(\alpha_p) \ar[d]^-{v_{\coker}} \\
\ker(\beta_p) \ar[r] & k(\Delta_p^{alt} E)^{\times} \ar[r]^-{\beta_p} & \Div(\Delta_p^{alt} E) \ar[r] & \Pic(\Delta_p^{alt}E) \\
}}
\label{}
\end{equation}

\medskip
First, injectivity of $v_{\coker}$ follows from a diagram chase and the fact that the middle square is cartesian. To establish surjectivity of $v_{\coker}$, let $t \in \Div(\Delta_p^{alt}E)$. We would like to lift $t$ to a good divisor on $\Delta_p^{alt}E$. In order to do so, we would like to wiggle $t$ by a principal divisor so that it meets $\Delta_p^{alt}E$ transversally for $q > p$. But since $\Delta_q^{alt}E = \varnothing$ for sufficiently large $q$ and each $\Delta_p^{alt}E$ has a a finite number of components, we may apply Bertini's theorem to find a lift $s \in \Div_g(\Delta_p^{alt}E)$ of $t$.
\end{proof}

\medskip
\noindent By replacing each column $k(\Delta_p^{alt}E)^{\times} \longrightarrow \Div(\Delta_p^{alt}E)$ with the quasi-isomorphic complex obtained from $\Div_g(\Delta_p^{alt}E)$ as in the lemma above, we can replace the diagram \eqref{simplicial E0 page} with the following diagram

\begin{equation}
\vcenter{\xymatrix{
\cdots \ar[r] & \Div_g(\Delta_{n-3}^{alt}E) \ar[r] & \Div_g(\Delta_{n-2}^{alt}E) \ar[r] & 0 \ar[r] & 0 \\
\cdots \ar[r] & A_{n-3} \ar[r] \ar[u] & A_{n-2} \ar[r] \ar[u] & k(\Delta_{n-1}^{alt}E)^{\times} \ar[u] \ar[r] & 0 \\
}}
\label{}
\end{equation}

\medskip
\noindent where all of the horizontal maps are indeed defined. We may then use this diagram to calculate the differential $d_2^{n-3,1}$ that appears in the spectral sequence \eqref{eq:cdh semisimplicial spectral sequence applied to K_m}. We claim this map is zero.

\begin{lem}
The differential $d_2^{n-3,1}$ appearing in the spectral sequence \eqref{eq:cdh semisimplicial spectral sequence applied to K_m} is the zero map.
\end{lem}

\begin{proof}
\medskip
\noindent Recall that $E$ was assumed to have $r$ irreducible components. We regret having to introduce the following notation. Let $\mathbf{r} = \{1 < \cdots < r\}$, and let us also set

\begin{equation}
I = \left\{\{i_0 < \cdots < i_{n-1}\} \,|\, i_1, \ldots, i_{n-1} \in \mathbf{r} \right\}.
\label{}
\end{equation}

\medskip
\noindent $I$ denotes the set of all ordered subsets of $\{1,\ldots,r\}$ that have length $n$. We will also let $\mathbf{i}$ and $\mathbf{j}$ denote ordered subsets of $\mathbf{m}$ of length $n-2$ and $n-1$, respectively. Keeping tight track of the indices would be a notational burden and detracts from the main thrust of the proof, so there will be some looseness in our usage of $\mathbf{i}$ and $\mathbf{j}$.

\medskip
\noindent If $\mathbf{i} = \{i_0,\ldots,i_{n-2}\}$ and $\mathbf{i} \varsubsetneq \mathbf{j} = \{i_0 \ldots, i_m, a, i_{m+1}, \ldots, i_{n-2}\}$ so that $\mathbf{j}$ is obtained from $\mathbf{i}$ by inserting $a$ after the $m^{\mathrm{th}}$ element of $\mathbf{i}$, then we define $\sign(\mathbf{i},\mathbf{j}) = (-1)^m$.

\medskip
Consider $D \in \Div_g(\Delta_{n-3}^{alt}E)$ that represents an element of $E_2^{n-3,1} = \ker(d_1^{n-3,1})$ (see \eqref{eq:simplicial spectral sequence E_1 page}). The image of $D$ in $\Pic(\Delta_{n-2}^{alt}E)$ is zero, so it pulls back to a rational function $g = (g_{\mathbf{j}}) \in A_{n-2}$.

\medskip
\noindent Write $D = (D_{\mathbf{i}})$ and $D_{\mathbf{i}} = D'_{\mathbf{i}} - D''_{\mathbf{i}}$ such that $D'_{\mathbf{i}}$ and $D''_{\mathbf{i}}$ are effective divisors whose supports intersect in codimension at least two. The divisors $D'_{\mathbf{i}}, D''_{\mathbf{i}}$ are defined locally on open $U$ by the vanishing of sections $f'_{\mathbf{i}}, f''_{\mathbf{i}} \in \Gamma(U,\mathcal{O}_U)$, respectively.

\medskip
\noindent On $E_{\mathbf{j}}$, the divisor $\sum_{\mathbf{i} \varsubsetneq \mathbf{j}} (-1)^{\sign(\mathbf{i},\mathbf{j})}(D_{\mathbf{i}} \cap E_{\mathbf{j}})$ has degree zero, and is locally defined by $f_{\mathbf{j}} \colon \!\!\!\! = \prod_{\mathbf{i} \varsubsetneq \mathbf{j}} (f'_{\mathbf{i}}/f''_{\mathbf{i}})^{(-1)^{\sign(\mathbf{i},\mathbf{j})}}$. Then the divisor defined locally by $g_{\mathbf{j}}/f_{\mathbf{j}}$ has no zeroes or poles hence is constant on $E_{\mathbf{j}}$. This shows that the function defined locally by the $f_{\mathbf{j}}$ is in fact actually a rational function, and gives the same divisor class as $g$. Since $D$ is a good divisor, it meets $\Delta_{n-1}^{alt}E$ transversally, i.e. the support of $D$ does not intersect $\Delta_{n-1}^{alt}E$. Then the zeros of the $f'_{\mathbf{i}}, f''_{\mathbf{i}}$ do not intersect $\Delta_{n-2}^{alt}E$, so we can use the $f_{\mathbf{j}}$ to evaluate $d_2^{n-3,1}(D)$. But then $d_2^{n-3,1}(D)$ is gotten by the composite of two face maps, so it must be trivial.
\end{proof}

\medskip
\noindent Consequently, $E_{\infty}^{n-1,0} = H^{n-1}(\mathcal{D}(E),k^{\times})$, and we have the following corollary.

\begin{cor} \label{cor: SES for H^{n-1}(E,G_m)}
Writing $\coker(\Pic)$ for the cokernel of $d_1^{n-3,1} \colon \Pic(\Delta_{n-3}^{alt}E) \longrightarrow \Pic(\Delta_{n-2}^{alt}E)$, we have a short exact sequence

\begin{equation} \label{H^{n-1}(E,G_m)SES}
\xymatrix{
0 \ar[r] & H^{n-1}(\mathcal{D}(E),k^{\times}) \ar[r] & H^{n-1}_{\cdh}(E,\mathbb{G}_m) \ar[r] & \coker(\Pic) \ar[r] & 0.
}
\end{equation}
\end{cor}

\begin{eg} \label{torus when k algebraically closed}
If $H_{n-2}(\mathcal{D}(E),\Z)$ is torsion-free, or if $k$ contains all roots of unity (e.g. when $k$ is algebraically closed), then we also have $H^{n-1}(\mathcal{D}(E),k^{\times}) = H^{n-1}(\mathcal{D}(E),\Z) \otimes k^{\times}$, via the universal coefficient theorem. In particular, $H^{n-1}(\mathcal{D}(E),k^{\times}) \cong (k^{\times})^r = T_E(k)$, for some $r$, is the $k$-points of some split torus $T_E$.
\end{eg}

\medskip
\noindent The thrust of the next subsection is to show that a $1$-motive naturally arises out of the spectral sequence \eqref{eq:cdh semisimplicial spectral sequence applied to K_m}.

%% file: 1-motive.tex
\section{Computation of Picard groups}
\label{sec:1-motive}

Since all of the schemes $\Delta_p^{alt}E$ are projective, the Picard functor is representable for these schemes; in particular, $\sPic^0(\Delta_p^{alt}E)$, the connected component of the Picard scheme $\sPic(\Delta_p^{alt}E)$, exists. Let the N\'{e}ron-Severi group functor, $\NS$, be the presheaf cokernel defined by $\sPic/\sPic^0$. As in Corollary \ref{cor: SES for H^{n-1}(E,G_m)}, whenever there is no ambiguity, we will write $\ker(\NS)$ for the kernel of the induced map on N\'{e}ron-Severi groups $\NS(\Delta_{n-3}^{alt}E) \longrightarrow \NS(\Delta_{n-2}^{alt}E)$, and similarly for the kernels and cokernels of other such maps induced by $\Delta_{n-3}^{alt}E \longrightarrow \Delta_{n-2}^{alt}E$.

\medskip
We are interested in the group $\coker(\Pic)$. Writing $\Pic$ as an extension of $\NS$ by $\Pic^0$ and then applying the snake lemma to the resulting diagram obtained from the map $\Delta_{n-3}^{alt}E \longrightarrow \Delta_{n-2}^{alt}E$ gives us an exact sequence ending in

\begin{equation} \label{extension of NS by Pic^0 diagram}
\xymatrix{
\cdots \ar[r] & \ker(\NS) \ar[r] & \coker(\Pic^0) \ar[r] & \coker(\Pic) \ar[r] & \coker(\NS) \ar[r] & 0.
}
\end{equation}

\medskip
\noindent Taking \eqref{H^{n-1}(E,G_m)SES} and pulling back along the map $\coker(\Pic^0) \stackrel{\beta}{\longrightarrow} \coker(\Pic)$ yields the diagram

\begin{equation} \label{diagram:pullback to get TGA}
\vcenter{\xymatrix{
0 \ar[r] & H^{n-1}(\mathcal{D}(E),k^{\times}) \ar[r] \ar@{=}[d] & G_E(k) \ar[r] \ar[d] & \coker(\Pic^0) \ar[d]^{\beta} \ar[r] & 0 \\
0 \ar[r] & H^{n-1}(\mathcal{D}(E),k^{\times}) \ar[r] & H^{n-1}_{\cdh}(E,\mathbb{G}_m) \ar[r] & \coker(\Pic) \ar[r] & 0 \\
}}
\end{equation}

\medskip
\noindent Applying the snake lemma to this diagram, we see that the two vertical maps on the right have the same kernel and cokernel, and that $\ker(\NS)$ surjects onto $\ker(\beta)$:

\begin{equation} \label{pullback_fg_ker_coker}
\vcenter{\xymatrix{
& & \ker(\NS) \ar@{>>}[d] \ar@{=}[r] & \ker(\NS) \ar@{>>}[d] \\
& & \ker(\beta) \ar@{=}[r] \ar@{>->}[d] & \ker(\beta) \ar@{>->}[d] \\
0 \ar[r] & H^{n-1}(\mathcal{D}(E),k^{\times}) \ar[r] \ar@{=}[d] & G_E(k) \ar[r] \ar[d] & \coker(\Pic^0) \ar[d]^-{\beta} \ar[r] & 0 \\
0 \ar[r] & H^{n-1}(\mathcal{D}(E),k^{\times}) \ar[r] & H^{n-1}_{\cdh}(E,\mathbb{G}_m) \ar[r] \ar@{>>}[d] & \coker(\Pic) \ar[r] \ar@{>>}[d] & 0 \\
& & \coker(\NS) \ar@{=}[r] & \coker(\NS) \\
}}
\end{equation}

\medskip
\noindent Furthermore, since $\NS(\Delta_{\bullet}^{alt}E)$ is a complex, the map $\NS(\Delta_{n-4}^{alt}E) \longrightarrow NS(\Delta_{n-3}^{alt}E)$ factors via $\ker(\NS)$. We note that the composite

\begin{equation} \label{map giving rise to M_E lattice}
\xymatrix{
\NS(\Delta_{n-4}^{alt}E) \ar[r] & \ker(\NS) \ar[r] & \coker(\Pic^0)
}
\end{equation}

\medskip
\noindent is zero; this follows immediately from the square

\begin{equation}
\vcenter{\xymatrix{
\Pic(\Delta_{n-4}^{alt}E) \ar@{>>}[r] \ar[d] & \NS(\Delta_{n-4}^{alt}E) \ar[d] \\
\ker(\Pic) \ar[r] & \ker(\NS)
}}
\label{}
\end{equation}

\medskip
For the rest of this section, let $k$ be algebraically closed and of sufficiently small cardinality so that there is an embedding $k \longrightarrow \C$. We will show that the diagrams

\begin{equation} \label{M'_E(k)}
\left[\vcenter{\xymatrix{
& & \ker(\NS) \ar[d] \\
0 \ar[r] & H^{n-1}(\mathcal{D}(E),k^{\times}) \ar[r] & G_E(k) \ar[r] & \coker(\Pic^0) \ar[r] & 0
}}
\right]
\end{equation}

\medskip
\noindent and

\begin{equation} \label{M_E(k)}
\left[\vcenter{\xymatrix{
& & H^{n-3}(\NS(\Delta_{\bullet}^{alt}E)) \ar[d] \\
0 \ar[r] & H^{n-1}(\mathcal{D}(E),k^{\times}) \ar[r] & G_E(k) \ar[r] & \coker(\Pic^0) \ar[r] & 0
}}
\right]
\end{equation}

\medskip
\noindent are isomorphic to the $k$-points of $1$-motives $M'_E$ and $M_E$, respectively, over $k$.

\medskip
Since $\Pic^0(\Delta_p^{alt}E)$ is representable and $k$ is algebraically closed, the functor of taking $k$-points is exact. In particular, the $k$-points of the cokernel of the map $\sPic^0(\Delta_{n-3}^{alt}E) \longrightarrow \sPic^0(\Delta_{n-2}^{alt}E)$ is the cokernel of the $\Pic^0$ groups, that is, $\coker(\Pic^0)$. In other words, $\coker(\Pic^0)$ is the $k$-points of the corresponding abelian variety $\coker(\sPic^0)$.

\medskip
Similarly, the group $H^{n-1}(\mathcal{D}(E),k^{\times}) \cong (k^{\times})^r$, is isomorphic to the $k$-points of a torus, as in Example \ref{torus when k algebraically closed}. Therefore, for ease of notation and for suggestiveness, let $T_E$ be a (split) torus so that $T_E(k) = H^{n-1}(\mathcal{D}(E),k^{\times})$.

\medskip
We may compose the map $G_E(k) \longrightarrow H^{n-1}_{\cdh}(E,\mathbb{G}_m)$ with the edge map $H^{n-1}_{\cdh}(E,\mathbb{G}_m) \longrightarrow KH_{2-n}(E)$, coming from the descent spectral sequence, to get a map $G_E(k) \longrightarrow H^{n-1}_{\cdh}(E,\mathbb{G}_m) \longrightarrow KH_{2-n}(E)$, which we will denote $\alpha$; we can write $\coker(\alpha)$ in the following short exact sequence:

\begin{equation} \label{cokeralpha, arbitrary n}
\xymatrix{
0 \ar[r] & \coker(\NS) \ar[r] & \coker(\alpha) \ar[r] & H^{n-2}(\mathcal{D}(E),\Z) \ar[r] & 0.
}
\end{equation}

\medskip
Similarly, $\ker(\alpha)$ can also be given by a short exact sequence:

\begin{equation} \label{keralpha, arbitrary n}
\xymatrix{
0 \ar[r] & \ker(\beta) \ar[r] & \ker(\alpha) \ar[r] & \ar[r] \im(d_2^{n-3,0}) & 0,
}
\end{equation}

\medskip
\noindent where $d_2^{n-3,0}$ is the $E_2$ differential $H_{\cdh}^{n-3}(E,\Z) \longrightarrow H^{n-1}_{\cdh}(E,\mathbb{G}_m)$ in the descent spectral sequence \eqref{eq:cdh-descent spectral sequence for KH}. Since the N\'{e}ron-Severi groups are finitely generated, $\ker(\NS)$ is finitely generated, so the quotient $\ker(\beta)$ is finitely generated as well. Furthermore, Lemma \ref{H^i(E,Z) finitely generated} tells us that the term $\im(d_2^{n-3,0})$ is also finitely generated, so $\ker(\alpha)$ is also finitely generated.

\medskip When $n = 3$, Lemma \ref{d_2^{0,0} is the zero map} implies that the edge map $H^2_{\cdh}(E,\mathbb{G}_m) \longrightarrow KH_{-1}(E)$ is an injection (see \eqref{eq: kh_ses, n = 3}), so $\ker(\alpha) = \ker(\beta)$. In particular, the sequence

\begin{equation} \label{}
\xymatrix{
\ker(\NS) \ar[r] & G_E(k) \ar[r] & KH_{-2}(X)
}
\end{equation}

\medskip
\noindent is exact. All in all, for general $n$, both $\ker(\alpha)$ and $\coker(\alpha)$ are finitely generated, so that $KH_{1-n}(X)$ is isomorphic to the $k$ points of some group scheme, up to some finitely generated groups.

\medskip
\noindent The rest of this section will be dedicated to showing the following.

\begin{prop}
The diagrams \eqref{M'_E(k)} and \eqref{M_E(k)} are isomorphic to the $k$-points of $1$-motives $M'_E$, $M_E$, respectively, over $k$.
\end{prop}

\begin{proof}
To the semisimplicial scheme $\Delta_{\bullet}^{alt}E$, we may associate a complex $C_{\bullet}(\Delta_{\bullet}^{alt}E)$ of schemes, following \cite[Sec. 2]{B-VRS03}. Make $Sch/k$ into an additive category by modifying the morphisms to be formal $\Z$-linear combinations of actual $k$-scheme morphisms, and then construct $C_{\bullet}(\Delta_{\bullet}^{alt}E)$ in the usual way, by taking the differentials to be alternating sums of face maps.

\medskip
For ease of notation, we will write $A_{\bullet} \colon\!\!\!\!= C_{\bullet}(\Delta_{\bullet}^{alt}E)$. We now check that our construction agrees with \cite{B-VRS03}. As it would be redundant to set up our own notation, we will merely follow theirs. We apply their construction to $X_{\bullet} = A_{\bullet}$. In addition, $\Delta_{\bullet}^{alt}E$ is already projective, so there is no need to take a compactification. So we have

\begin{equation}
\begin{split}
\leftexp{o}{W'^0(A_{\bullet})} & = A_{\bullet} \\
\leftexp{o}{W'^1(A_{\bullet})} & = E_n \longrightarrow E_{n-1} \longrightarrow \cdots \longrightarrow E_2 \longrightarrow E_1 \\
& \,\,\, \vdots  \\
\leftexp{o}{W'^{n-1}(A_{\bullet})} & = E_n \longrightarrow E_{n-1} \\
\leftexp{o}{W'^n(A_{\bullet})} & = E_n\\
\leftexp{o}{W'^{n+1}(A_{\bullet})} & = \varnothing
\label{}
\end{split}
\end{equation}

\noindent and

\begin{equation}
\begin{split}
\leftexp{o}{W''^{-1}(A_{\bullet})} & = A_{\bullet} \\
\leftexp{o}{W''^0(A_{\bullet})}    & = A_{\bullet} \\
\leftexp{o}{W''^1(A_{\bullet})}    & = \varnothing
\label{}
\end{split}
\end{equation}

\noindent so that $W$, the convolution of $W'$ and $W''$, is given by

\begin{equation}
\begin{split}
\leftexp{o}{W''^{-1}(A_{\bullet})} & = \Delta_{\bullet}^{alt}E \\
\leftexp{o}{W''^0(A_{\bullet})}    & = \Delta_{\bullet}^{alt}E \\
\leftexp{o}{W'^1(A_{\bullet})} & = E_n \longrightarrow E_{n-1} \longrightarrow \cdots \longrightarrow E_2 \longrightarrow E_1 \\
& \,\,\, \vdots \\
\leftexp{o}{W'^{n-1}(A_{\bullet})} & = E_n \longrightarrow E_{n-1} \\
\leftexp{o}{W'^n(A_{\bullet})} & = E_n\\
\leftexp{o}{W'^{n+1}(A_{\bullet})} & = \varnothing
\label{}
\end{split}
\end{equation}

\noindent where the chain maps are the alternating sum of the face maps. In addition, in any explicitly written-out complexes, the leftmost term has homological degree zero. Then the spectral sequence \cite[3.1.3]{B-VRS03} with $r = 0$ is

\begin{equation}
E_1^{p,q} = H^q(\Delta_p^{alt}E,\mathbb{G}_m) \Longrightarrow H^{p+q}(\mathcal{K}'),
\label{B-VRSss}
\end{equation}

\noindent which is the spectral sequence \eqref{eq:cdh semisimplicial spectral sequence applied to K_m}. Next, we claim that the $1$-motives $M_{n-1}'(A_{\bullet}) = [\Gamma'_{n-1}(A_{\bullet}) \longrightarrow G_{n-1}(A_{\bullet})]$ and $M_{n-1}'(A_{\bullet}) = [\Gamma'_{n-1}(A_{\bullet}) \longrightarrow G_{n-1}(A_{\bullet})]$ are the $1$-motives $M'_E$ and $M_E$ referred to earlier, where

\begin{equation}
\begin{split}
\Gamma_{n-1}'(A_{\bullet}) & = \ker\left(\partial \colon P_{n-3}(A_{\bullet})/P_{n-3}(A_{\bullet})^0 \longrightarrow P_{\ge n-2}(A_{\bullet})/P_{\ge n-2}(A_{\bullet})^0\right) \\
\Gamma_{n-1}(A_{\bullet}) & = \coker\left(\NS(\Delta_{n-4}^{alt}E) \longrightarrow \Gamma_{n-1}'(A_{\bullet})\right)
\end{split}
\label{defn of gamma lattices}
\end{equation}

\medskip
\noindent and

\begin{equation}
G_{n-1}(A_{\bullet}) = \coker\left(\partial \colon P_{n-3}(A_{\bullet}) \longrightarrow P_{\ge n-2}(A_{\bullet})\right),
\label{}
\end{equation}

\medskip
\noindent where

\begin{equation}
\begin{split}
P_{\ge i}(A_{\bullet}) & = H^{i+1}(W^i\mathcal{K}') \\
P_i(A_{\bullet}) & = H^{i+1}(Gr_W^i\mathcal{K}').
\label{}
\end{split}
\end{equation}

\medskip
\noindent and $\mathcal{K}'$ is as in spectral sequence \eqref{B-VRSss}. To check this, we first compute the lattices $\Gamma'_{n-1}(A_{\bullet})$ and $\Gamma_{n-1}(A_{\bullet})$. We can see that $P_{n-3}(A_{\bullet}) = \Pic(\Delta_{n-3}^{alt}E)$; we still need to determine $P_{\ge n-2}(A_{\bullet})$. For the latter, the above spectral sequence \eqref{B-VRSss} gives a short exact sequence

\begin{equation}
\xymatrix{
0 \ar[r] & H^{n-1}(\mathcal{D}(E),k^{\times}) \ar[r] & P_{\ge n-2}(A_{\bullet}) \ar[r] & \Pic(\Delta_{n-2}^{alt}E) \ar[r] & 0.
}
\label{Pge1ses}
\end{equation}

\medskip
\noindent Consider the pullback of the diagram along the inclusion $\Pic^0(\Delta_{n-2}^{alt}E) \longrightarrow \Pic(\Delta_{n-2}^{alt}E)$. The pullback of this square is $P_{\ge n-2}(A_{\bullet})^0$ \cite[Lemma 3.3]{B-VRS03}. We indicate this in the diagram below.

\begin{equation}
\vcenter{\xymatrix{
0 \ar[r] & H^{n-1}(\mathcal{D}(E),k^{\times}) \ar[r] \ar@{=}[d] & P_{\ge n-2}(A_{\bullet})^0 \ar[r] \ar[d] & \Pic^0(\Delta_{n-2}^{alt}E) \ar[r] \ar[d] & 0 \\
0 \ar[r] & H^{n-1}(\mathcal{D}(E),k^{\times}) \ar[r]            & P_{\ge n-2}(A_{\bullet}) \ar[r]          & \Pic(\Delta_{n-2}^{alt}E) \ar[r]          & 0
}}
\label{}
\end{equation}

\medskip
\noindent Applying the snake lemma to the above diagram, we obtain

\begin{equation}
P_{\ge n-2}(A_{\bullet})/P_{\ge n-2}(A_{\bullet})^0 \cong \NS(\Delta_{n-2}^{alt}E),
\label{}
\end{equation}

\medskip
\noindent so that $\Gamma_{n-1}'(A_{\bullet}) = \ker(\NS(\Delta_{n-3}^{alt}E) \longrightarrow \NS(\Delta_{n-2}^{alt}E))$, which agrees with our lattice $L_E = \ker(\NS)$. The lattice $\Gamma_{n-1}(A_{\bullet})$ is just the cokernel

\begin{equation}
\begin{split}
\Gamma_{n-1}(A_{\bullet}) & = \coker(\NS(\Delta_{n-4}^{alt}E) \longrightarrow \Gamma_{n-1}'(A_{\bullet})) \\
                          & = H^{n-3}(\NS(A_{\bullet}))
\end{split}
\label{}
\end{equation}

\medskip
\noindent which agrees with our other lattice term in \eqref{M_E(k)}. It remains to check that the semiabelian variety $G_{n-1}(A_{\bullet})$ agrees with our $G_E$. Using the short exact sequence above that calculates $P_{\ge n-2}(A_{\bullet})$, we get

\begin{equation}
\vcenter{\xymatrix{
 & & \Pic^0(\Delta_{n-3}^{alt}E) \ar[d]^-g \\
0 \ar[r] & T_E \ar[r] & P_{\ge n-2}(A_{\bullet})^0 \ar[r] & \Pic^0(\Delta_{n-2}^{alt}E) \ar[r] & 0
}}
\label{1-motive before modding out}
\end{equation}

\medskip
\noindent where $G_{n-1}(A_{\bullet})$ is the cokernel of the vertical map $g$. We take the pullback of the first horizontal map $T_E \longrightarrow P_{\ge n-2}(A_{\bullet})^0$ along $g$.

\begin{equation} \label{diagram:B-VRS03 group scheme analysis}
\vcenter{\xymatrix{
& & & 0 \ar[d] \\
0 \ar[r] & W \ar[r] \ar[d]^-f & \Pic^0(\Delta_{n-3}^{alt}E) \ar[d]^-g \ar[r] & \Pic^0(\Delta_{n-3}^{alt}E)/W \ar[r] \ar[d]^-h & 0  \\
0 \ar[r] & T_E \ar[r] \ar[d] & P_{\ge n-2}(A_{\bullet})^0 \ar[r] \ar[d] & \Pic^0(\Delta_{n-2}^{alt}E) \ar[r] \ar[d] & 0 \\
 & \coker f \ar[r] & G_{n-1}(A_{\bullet}) \ar[r] & \coker h \ar[r] & 0
}}
\end{equation}

\medskip
\noindent Since $T_E$ is a closed subgroup of $P_{\ge n-2}(A_\bullet)^0$, we see that $W$ is a closed subgroup of $\Pic^0(\Delta_{n-3}^{alt}E)$. Furthermore, because the square is Cartesian, the induced map on cokernels is injective. We add these observations to the diagram \eqref{1-motive before modding out}. To finish, we need the following lemma:

\medskip
\begin{lem}
The $k$-points of the bottom row of \eqref{diagram:B-VRS03 group scheme analysis} isomorphic to the short exact sequence in the top row of the diagram \eqref{diagram:pullback to get TGA}.
\end{lem}

\medskip
\begin{proof}
\noindent Let $W' = \im f$ denote the image of $W$ in $T_E$. Since $\Pic^0(\Delta_{n-3}^{alt}E)$ is proper over $k$, so $g$ is also proper. We have already observed that $W$ is proper over $k$ as well. Furthermore, the map $W \longrightarrow T_E$ is also proper, so $W'$ is a closed subvariety of $T_E$ that is proper over $k$ \cite[II, Exercise 4.4]{Har77}. On the other hand, $T_E$ is affine, and $W'$, being closed in $T_E$, is also affine. But then $W'$ is finite over $k$, as it is affine and proper over $k$\cite[II, Exercise 4.6]{Har77}.

\medskip
\noindent In addition, since $W'$ is a finite subgroup of $T_E$, we claim that $\coker f$ is isomorphic to $T_E$. $T_E$ is a group of multiplicative type, and since all finite subgroups of a group of multiplicative type are also of multiplicative type, $W'$ is of multiplicative type \cite[2.2]{Wat79}. There is an anti-equivalence between group schemes of multiplicative type over $k$ and finite abelian groups \cite[Proposition 20.17]{BOI}. Here, the map $W' \longrightarrow T_E$ corresponds to a surjective map of a lattice onto a finite abelian group. The kernel of this map must also be finitely generated free abelian of the same rank, so that $\coker f$ must be isomorphic to a copy of $T_E$.

\medskip
\noindent Finally, since the top right horizontal map is surjective, $\coker h$ is the same as the cokernel of the map $\Pic^0(\Delta_{n-3}^{alt}E) \longrightarrow \Pic^0(\Delta_{n-2}^{alt}E)$, as in our case.
\end{proof}

\medskip
\noindent Applying the snake lemma and making the identification $\coker f \cong T_E$ yields a short exact sequence of commutative group schemes

\begin{equation}
\xymatrix{
0 \ar[r] & T_E \ar[r] & G_{n-1}(A_{\bullet}) \ar[r] & \coker(\Pic^0) \ar[r] & 0,
}
\end{equation}

\noindent which agrees with our construction.
\end{proof}

\medskip
\noindent Now that we have established that $M_E$ is a $1$-motive, we are interested in how to calculate it. In the landmark paper \cite{Del74}, Deligne established an equivalence between torsion-free $1$-motives and torsion-free mixed Hodge structures of a given type; we state the version given in \cite[1.5]{B-VRS03}.

\begin{thm} \label{thm:deligne's equivalence}
Let $\mathcal{M}_1(\C)$ denote the category of $1$-motives over $\C$, and let $\text{MHS}_1$ denote the category of mixed Hodge structures of type $\{(0,0),(0,1),(1,0),(1,1)\}$, such that $Gr_1^WH$ is polarizable. Then we have an equivalence of categories

\begin{equation} \label{eq:MHS, 1-motive equivalence}
r_{\mathcal{H}} \colon \mathcal{M}_1 \longrightarrow \text{MHS}_1.
\end{equation}
\end{thm}

\medskip
\noindent The main theorem \cite[Theorem 0.1]{B-VRS03} asserts that the free part $1$-motive $(M_E)_{\mathrm{fr}}$, after base extending to $\C$, corresponds to a unique mixed Hodge structure $H_E$ in $W_2H^{n-1}(E(\C),\Z)$. (More specifically, $H_E$ is the unique largest torsion-free mixed Hodge structure of type $\{(0,0),(0,1),(1,0),(1,1)\}$ in $W_2H^{n-1}(E(\C),\Z)$ such that $\Gr_1^W\!\!H_E$ is polarizable.) This gives us a concrete way of computing the free part of the $1$-motive $M_E$ arising from the computation of $H^{n-1}_{\cdh}(E,\mathbb{G}_m)$.

%% file: independence.tex
\section{Independence of the choice of resolution}
\label{sec:independence of resolution}

Now that we have constructed a $1$-motive $M_E = [L_E \longrightarrow G_E]$, we wish to determine to what extent it is independent of the choice of resolution. Under Deligne's equivalence of $1$-motives and mixed Hodge structures, we get another $1$-motive, which we denote $M = [L \longrightarrow G]$, that comes from a unique mixed Hodge structure $H$ in $W_2H^n(X(\C),\Z)$, of the considered type. We will not only establish to what extent $M_E$ is independent of the choice of resolution, but also we will establish a relation between $M_E$ and $M$. The precise statement is given below.

\begin{prop} \label{1-motive independence of resolution}
For each resolution $p \colon \tilde{X} \longrightarrow X$, there exists a morphism $M_E \longrightarrow M$ which is an isomorphism unless $n = 3$, in which case we have an isomorphism on the non-lattice parts and a surjection on the lattices.
\end{prop}

\begin{proof}
Taking the long exact sequence in singular cohomology (of the $\C$-points) induced by the blowup square resolving the singularities of $X$ via $p$, we obtain

\begin{equation} \label{singular cohomology les}
\xymatrix{
\cdots \ar[r] & H^{r-1}(\tilde{X},\Z) \oplus H^{r-1}(Z,\Z) \ar[r] & H^{r-1}(E,\Z) \ar[r] & H^r(X,\Z) \ar[r] & \cdots
}
\end{equation}

\noindent From this long exact sequence, we get a map $H_E \longrightarrow H$ of mixed Hodge structures, since the weights are functorial with respect to morphisms. Since the groups $H^i(Z,\Z)$ vanish for $i > n-2$ and the groups $H^i(\tilde{X},\Z)$ are pure of weight $i$, and $n \ge 3$, taking the weight $2$ part of the sequence yields an isomorphism $W_2H^{n-1}(E,\Z) \cong W_2H^n(X,\Z)$ unless $n = 3$, in which case we only have a surjection. Similarly, taking the weight $1$ part of the above sequence yields an isomorphism $W_1H^{n-1}(E,\Z) \cong W_1H^n(X,\Z)$. The weight $2$ part contains the lattice, and the weight $1$ part contains the rest of the $1$-motive, proving the claim.
\end{proof}

\begin{rem}
Because the map $L_E \longrightarrow G$ factors through $L$, we see from the composite

\begin{equation}
\xymatrix{
L_E(k) \ar@{>>}[r] & L(k) \ar[r] & G(k) \ar[r] & KH_{1-n}(E)
}
\label{}
\end{equation}

\medskip
\noindent that the images of $L_E$ and $L$ in $G$ are the same. So when $n = 3$, the sequence

\begin{equation}
\xymatrix{
L(k) \ar[r] & G(k) \ar[r] & KH_{-2}(E)
}
\label{eq:exactness_at_G(k)}
\end{equation}

\medskip
\noindent is still exact.
\end{rem}

\medskip
\noindent Another way to see that the torus $H^{n-1}(\mathcal{D}(E),k^{\times})$ is independent of the resolution is to see that the homotopy type of $\mathcal{D}(E)$ is independent of the choice of resolution \cite{Ste08}. So all of the cohomology groups $H^i(\mathcal{D}(E),\Z)$ (in particular, $i = n-3, n-2$ coming out of the exact sequence \eqref{eq:kh es, arbitrary n}) are independent of the choice of resolution, and thus $H^{n-1}_{\cdh}(E,\mathbb{G}_m)$ is independent of the choice of resolution as well. More directly, we can apply \emph{cdh}-descent to the cohomology groups themselves; we get a long exact sequence

\begin{equation}
\xymatrix{
\cdots \ar[r] & H_{\cdh}^{n-1}(Z,\mathbb{G}_m) \oplus H_{\cdh}^{n-1}(\tilde{X},\mathbb{G}_m) \ar[r] & H_{\cdh}^{n-1}(E,\mathbb{G}_m) \ar[r] & H_{\cdh}^n(X,\mathbb{G}_m) \ar[r] & \cdots
}
\label{}
\end{equation}

\medskip
\noindent Since $Z$ and $\tilde{X}$ are smooth, their \emph{cdh}-cohomology groups agree with their Zariski cohomology groups:

\begin{thm} \label{lem:cdh-Zar-cohomology-isomorphism}
Let $Y$ be smooth over $k$, and $\mathscr{F}$ a homotopy invariant sheaf with transfers on the \emph{cdh}-site over $X$. Then the change of topology morphism induces an isomorphism $H^p_{\cdh}(X,\mathscr{F}) \cong H^p_{\Zar}(X,\mathscr{F})$.
\end{thm}

\begin{proof}
\cite{VSF00}
\end{proof}

\medskip
\noindent This result is quite useful, because on smooth schemes, all of the sheaves $a_{\Zar}K_n$ are homotopy invariant sheaves with transfers \cite[Section 3.4]{Voe00}. Furthermore, $H^i_{\Zar}(Y,\mathbb{G}_m) = 0$ whenever $Y$ is smooth over $k$ and $i > 1$, so we obtain an isomorphism $H_{\cdh}^{n-1}(E,\mathbb{G}_m) \cong H_{\cdh}^n(X,\mathbb{G}_m)$.

\medskip
\noindent Now that we know that $H_{\cdh}^{n-1}(E,\mathbb{G}_m)$ is independent of the choice of resolution, the exact sequence \eqref{H^{n-1}(E,G_m)SES} shows that the group $\coker(\Pic)$ is also independent of the choice of resolution. Furthermore, since $\coker(\Pic^0)$ was independent of the choice of resolution, the cokernel of $\coker(\Pic^0) \longrightarrow \coker(\Pic)$, which is $\coker(\NS)$, is also independent of the choice of resolution, as is the kernel of that map. In summary, all of the various groups appearing in the diagram \eqref{extension of NS by Pic^0 diagram} are independent of the choice of resolution except \emph{possibly} the group $\ker(\NS)$, and only in the case $n = 3$. We give an example to show that indeed this is the case, that $\ker(\NS)$ is not independent of the choice of resolution when $n = 3$.

\begin{eg}
Let $X$ be an integral $3$-fold $X$ with a smooth singular locus $Z$ of dimension $\le 1$. Suppose we have an excellent resolution $p \colon \tilde{X} \longrightarrow X$ with exceptional divisor $E$ that has at least two irreducible components $E_1, E_2$ that have a nonempty intersection $E_{12}$ (which, by assumption, must be a smooth curve). Let the other irreducible components of $E$ be $E_3, \ldots, E_m$. Take a closed point $x$ that lies in $E_{12}$ but does not lie in any of the other $E_i$. Blowing up along $x$, we obtain a diagram

\begin{equation}
\vcenter{\xymatrix{
\Bl_x\!E \ar[r] \ar[d] & \Bl_x\!\tilde{X} \ar[d] \\
E \ar[r] \ar[d] & \tilde{X} \ar[d] \\
Z \ar[r] & X
}}
\label{}
\end{equation}

\medskip
\noindent so that $\Bl_{x}\!\tilde{X} \longrightarrow X$ is also an excellent resolution. $\Bl_{x}\!E$ then has $m+1$ irreducible components: the two blown-up components $E'_1 = \Bl_x\!E_1$ and $E'_2 = \Bl_x\!E_2$; the ``untouched'' components $E_3, \ldots, E_m$; and a new component $E'$ that is the exceptional divisor of $\Bl_x\!\tilde{X}$. The relationships between the intersections of the various components are given below.

\begin{equation}
\begin{aligned}
E'_1 \cap E'_2 & = \Bl_x\!E_{12} \cong E_{12} \\
E'_i \cap E & = \mbox{exceptional divisor of } \Bl_x\!E_i, \longrightarrow E_i & i = 1, 2 \\
E_i \cap E' & = \varnothing & i > 2 \\
E_i \cap E'_j & \cong E_i \cap E_j & i > 2, j = 1, 2
\label{}
\end{aligned}
\end{equation}

\medskip
\noindent In general, for a smooth surface $S$ that contains a point $y$, we will have $\NS(\Bl_y\!S) = \NS(S) \oplus \Z$ \cite[V, Theorem 5.8]{Har77}, so that from the following diagram obtained from blowing up along $x$

\begin{equation}
\vcenter{\xymatrix{
\NS(\Delta_0^{alt}E) \ar[r] \ar[d] & \NS(\Delta_0^{alt}E)  \oplus \NS(E') \oplus \Z^2  \ar[d] \\
\NS(\Delta_1^{alt}E) \ar[r] & \NS(\Delta_1^{alt}E) \oplus \Z^2
}}
\label{}
\end{equation}

\medskip
\noindent we see that $\NS(\Delta_0^{alt}E)$ has changed by $\NS(E') \oplus \Z^2$ and that $\NS(\Delta_1^{alt}E)$ has changed by $\Z^2$. Since $E'$ is projective, $\NS(E')$ has rank at least $1$, so that $\ker(\NS)$ must become strictly bigger, and in particular depends on the choice of resolution of $X$.
\label{example of ker(NS) not independent of choice of resolution}
\end{eg}

\medskip
\noindent This makes sense, because by Proposition \ref{1-motive independence of resolution}, we have in general only a surjection $H_{n-3}(NS(\Delta_{\bullet}^{alt}E)) \longrightarrow L$ and not an isomorphism.

\begin{rem}
Proposition \ref{1-motive independence of resolution} tells us that when $X$ is projective, the $1$-motive $M_E$ is independent of the choice of good resolution $p$ unless $n = 3$, in which case the non-lattice parts of the $1$-motive are independent of the choice of good resolution $p$. Therefore, to calculate $KH_{1-n}(X)$ when $X$ is not projective, we need only take an algebraic compactification $\overline{X}$ of $X$, smooth along the boundary, and \emph{then} compute $KH_{1-n}(\overline{X})$, as $KH_{1-n}(X) \cong KH_{2-n}(E) \cong KH_{1-n}(\overline{X})$. This shows that $KH_{1-n}(X)$ is independent of the choice of algebraic compactification $\overline{X}$. This result makes sense in light of the observation that negative $KH$ vanishes for smooth schemes, and we compactify away from the singular locus. In some sense, we are computing, $KH_{1-n}$ of the singularity $x_0$ locally sitting inside $X$.
\end{rem}

\medskip
\noindent We wrap things up by putting together everything we have proven so far.

\begin{thm}[Main Theorem for $KH_{1-n}(X)$] \label{thm:main theorem for KH_{1-n}(X)}
Let $X$ be an normal, integral $n$-fold over an algebraically closed field $k$ of characteristic zero, with singular locus $Z = \Sing(X)$ such that $Z$ is smooth or $\codim Z > 2$. Then there exists a $1$-motive

\begin{equation} \label{}
M = \left[\vcenter{ \xymatrix{& & L \ar[d] \\
0 \ar[r] & T \ar[r] & G \ar[r] & A \ar[r] & 0
}}\right]
\end{equation}

\medskip
\noindent and a map $\alpha \colon G(k)\longrightarrow KH_{1-n}(X)$, natural in $X$, whose kernel and cokernel are finitely generated. If $p \colon \tilde{X} \longrightarrow X$ is any good resolution of singularities, then $\ker(\alpha)$ and $\coker(\alpha)$ have the more explicit descriptions \eqref{keralpha, arbitrary n} and \eqref{cokeralpha, arbitrary n}, respectively. In particular, the descriptions of $\ker(\alpha)$ and $\ker(\beta)$ are independent of the choice of resolution of $X$.

\medskip
\noindent Furthermore, if $X \longrightarrow \overline{X}$ is an algebraic compactification of $X$, then after base extending to $\C$, the (torsion-free) $1$-motive $(M_{\C})_{\mathrm{fr}}$ corresponds, under the equivalence \eqref{eq:MHS, 1-motive equivalence}, to the unique largest torsion-free mixed Hodge structure $H$ of type $\{(0,0),(0,1),(1,0),(1,1)\}$ in $W_2H^n(\overline{X}(\C),\Z)$ such that $\Gr_1^W\!\!H$ is polarizable. Moreover, the non-lattice parts of $M$, and hence the map $\alpha$, are independent of the choice of algebraic compactification $X \longrightarrow \overline{X}$.

\medskip
\noindent Finally, when $n = 3$, then we have the additional property that the sequence \eqref{eq:exactness_at_G(k)} is exact.
\end{thm}

%% file: NK.tex

\section{Calculation of $NK_{1-n}(X)$}
\label{sec:NK}

We now turn our attention towards $NK_{1-n}(X)$, the other remaining contribution to $K_{1-n}(X)$. For this section, let $k$ be a field of characteristic zero (not necessarily algebraically closed) and $X$ be a (not necessarily irreducible) $n$-dimensional variety over $k$ with isolated singularity $x_0$. We first establish the exact sequence \eqref{eq:exact sequence calculating K_n} referred to in the introduction.

\begin{lem} \label{lem:K_{1-n}(X) exact sequence}
There is an exact sequence

\begin{equation}
\xymatrix{
NK_{1-n}(X) \ar[r]^-{d_1^{1,1-n}} & K_{1-n}(X) \ar[r] & KH_{1-n}(X) \ar[r] & 0.
}
\label{K_{1-n}(X) exact sequence}
\end{equation}
\end{lem}

\begin{proof}
There is a strongly convergent, homological right half-plane spectral sequence \cite{Wei89}

\begin{equation} \label{eq: NK spectral sequence abutting to KH}
E^1_{p,q} = N^pK_q(X) \Longrightarrow KH_{p+q}(X)
\end{equation}

\medskip
The $K$-dimension theorem \cite[Conjecture 0.1]{CHSW08} implies that the groups $N^pK_{-q}(X)$ are zero whenever $q \ge n$ and $p \ge 1$.

\medskip
\noindent So we can see that there are no nonzero differentials coming into or going out of each $E_m^{0,1-n}$ after the first page, so that $E^{\infty}_{0,1-n} = E^2_{0,1-n}$. In addition, all of the groups $E^{\infty}_{p,1-n-p}$ are zero, except when $p = 0$. This gives us the exact sequence we are looking for.
\end{proof}

\medskip
We now reduce to the case when $X$ is affine.

\begin{lem} \label{N^tK_{-q} independent of open nbhd}
$N^tK_{-q}(X) \cong N^tK_{-q}(U)$ for any $q \in \Z$, any $t \ge 1$, and any open $U \subset X$ containing the isolated singularity $x_0$.
\end{lem}

\begin{proof}
We have a spectral sequence \cite[Theorem 10.3]{TT90}

\begin{equation} \label{eq:spectral sequence abutting to N^tK_q}
E_2^{p,q} = H_{\Zar}^p(X,a_{\Zar}N^tK_q) \Longrightarrow N^tK_{-p-q}(X)
\end{equation}

\medskip
We apply the spectral sequence \eqref{eq:spectral sequence abutting to N^tK_q} to $X$. Because smooth schemes are $K_{-q}$-regular, it follows that for any smooth open subscheme $U \subset X$, we have $N^tK_{-q}(U) = 0$ whenever $t \ge 1$, as we have indicated above. Since $X$ has only a singularity at $x_0$, we have $N^tK_{-q}(U) = 0$ whenever $x_0 \notin U$. It follows that the Zariski sheaf $aN^tK_{-q}$ is a skyscraper sheaf supported at $x_0$. In particular, $aN^tK_{-q}$ is flasque, so it has no higher cohomologies. Since $E_2^{p,q} = 0$ for $p \ne 0$, all differentials are zero, so we conclude that $E_2^{p,q} = E_{\infty}^{p,q}$, and thus $H_{\Zar}^0(X,aN^tK_{-q}) \cong N^tK_{-q}(X)$. But since $aN^tK_{-q}$ is a skyscraper sheaf, we have $(aN^tK_{-q})(U) = (aN^tK_{-q})(X)$, proving the claim.
\end{proof}

\medskip
\noindent In particular, we may choose $U = \Spec R$ to be an open affine neighborhood of $x_0$. The intuition here is that since the $N^tK_{-q}$-groups are zero on smooth schemes, they only detect singularities, and their value depends only on the type of singularity involved.

\medskip
\noindent Recall that we are interested in the case $q = n-1$. Cortin\~{a}s, et al. \cite[Example 3.5, Proposition 4.1]{CHWW10a} elucidates the structure of the $N^pK_q$ groups, which, specializing to $p = 1$ and $q = n-1$, gives

\begin{equation}
NK_{1-n}(X) \cong NK_{1-n}(U) \cong H^{n-1}_{\cdh}(U,\mathcal{O}) \otimes_{\Q} t\Q[t].
\label{NK_{1-n}(U) description}
\end{equation}

\medskip
\noindent

\medskip
\noindent The maps in the spectral sequence \eqref{eq: NK spectral sequence abutting to KH} are induced by the maps on the simplicial structure of $X \times \mathbb{A}^{\bullet}$; in particular,

\begin{equation} \label{}
\xymatrix{
NK_{-q}(X) = \ker(\partial_0 : K_{-q}(X \times \mathbb{A}^1) \ar[r]^-{t = 0} & K_{-q}(X)),}
\end{equation}

\medskip
\noindent where $t$ is the parameter of $\mathbb{A}^1$ \--- the same $t$ as in \eqref{NK_{1-n}(U) description}. The decomposition \eqref{NK_{1-n}(U) description}, found in \cite{CHWW10a}, boils down to applying the K\"{u}nneth formula for Hochschild homology \cite[Proposition 9.4.1]{WHo94}

\begin{equation} \label{eq:Kunneth formula for HH}
HH_n(R[t]) \cong \oplus_{i+j=n} HH_i(R) \otimes_{\Q} HH_j(\Q[t]),
\end{equation}

\medskip
\noindent from which we see that the $t$ in the $\Q[t]$ is indeed the parameter $t$ in the copy of $\mathbb{A}^1$ when computing the $N$-functors.

\medskip
\noindent The differential $\partial_0 - \partial_1 \colon K_{1-n}(X \times \mathbb{A}^1) \longrightarrow K_{1-n}(X)$ reduces to just $-\partial_1$ on $NK_{1-n}(U) = \ker(\partial_0)$, and $\partial_1$ just sets $t = 1$. Therefore, the image of $NK_{1-n}(X)$ in $K_{1-n}(X)$ is isomorphic to $H^{n-1}_{\cdh}(U,\mathcal{O})$. In summary, we have proven that

\begin{prop} \label{prop:SES describing K_{1-n}(X)}
There is a short exact sequence

\begin{equation} \label{SES describing K_{1-n}(X)}
\xymatrix{
0 \ar[r] & H^{n-1}_{\cdh}(U,\mathcal{O}) \ar[r] & K_{1-n}(X) \ar[r] & KH_{1-n}(X) \ar[r] & 0}.
\end{equation}
\end{prop}

\begin{rem}
The observation here that the maps in the spectral sequence come from the simplicial structure on $X \times \mathbb{A}^{\bullet}$ can be taken further. For example, we can say something about $K_{2-n}(X)$. Proceeding as in the computation of $NK_{1-n}(X)$, we have, via \cite[Corollary 4.2]{CHWW10a},

\begin{equation}
\begin{split}
N^2K_{1-n}(U) & \cong NK_{1-n}(U) \otimes_{\Q} s_1\Q[s_1] \\
             & \cong H^{n-1}_{\cdh}(U,\mathcal{O}) \otimes_{\Q} s_0\Q[s_0] \otimes_{\Q} s_1\Q[s_1].
\label{}
\end{split}
\end{equation}

\medskip
\noindent The top face map from $K_{1-n}(X \times \mathbb{A}^2) \longrightarrow K_{1-n}(X \times \mathbb{A}^1)$ sends $1-s_0-s_1$ to zero, so it sends $s_0$ to $t$ and $s_1$ to $1-t$. Therefore, the image of $d_1^{2,1-n}$ in $NK_{1-n}(X)$ is just $H^{n-1}_{\cdh}(U,\mathcal{O}) \otimes_{\Q} t(1-t)\Q[t]$, which is precisely the kernel of the map $\partial_1 = (t \mapsto 1)$. The $E_1$ page of the spectral sequence is therefore exact at $(1,1-n)$, and so $E_2^{1,1-n} = 0$. We may make the same argument for the map $d_1^{3,1-n} : N^3K_{1-n}(X) \longrightarrow N^2K_{1-n}(X)$. Let us write

\begin{equation}
\begin{split}
N^3K_{1-n}(X) & \cong NK_{1-n}(X) \otimes r_1\Q[r_1] \otimes r_2\Q[r_2] \\
              & \cong H^{n-1}_{\cdh}(U,\mathcal{O}) \otimes r_0\Q[r_0] \otimes r_1\Q[r_1] \otimes r_2\Q[r_2] \\
N^2K_{1-n}(X) & \cong H^{n-1}_{\cdh}(U,\mathcal{O}) \otimes s_0\Q[s_0] \otimes s_1\Q[s_1].
\label{}
\end{split}
\end{equation}

\medskip
\noindent The differential coming out of $N^3K_{1-n}(X)$ is just the one that sends $1-r_0-r_1-r_2$ to $0$, so $r_0 \mapsto s_0$, $r_1 \mapsto s_1$, $r_2 \mapsto 1-s_0-s_1$, so the image of this map is just

\begin{equation} \label{}
(1-s_0-s_1) (H^{n-1}_{\cdh}(U,\mathcal{O}) \otimes s_0\Q[s_0] \otimes s_1\Q[s_1]) = H^{n-1}_{\cdh}(U,\mathcal{O}) \otimes s_0s_1(1-s_0-s_1)\Q[s_0,s_1],
\end{equation}

\medskip
\noindent which is precisely the kernel of the map $\partial_2 = d_1^{2,1-n}$. Thus $0 = E_3^{2,1-n} = E_{\infty}^{2,1-n}$, and we conclude that we have an exact sequence

\begin{equation}
\xymatrix{
NK_{2-n}(X) \ar[r] & K_{2-n}(X) \ar[r] & KH_{2-n}(X) \ar[r] & 0.
}
\label{}
\end{equation}

\medskip
\noindent In particular, the map $K_{2-n}(X) \longrightarrow KH_{2-n}(X)$ is surjective.
\end{rem}

\medskip
\noindent As we have already noted, the group $H^{n-1}_{\cdh}(U,\mathcal{O})$ in \eqref{SES describing K_{1-n}(X)} is independent of the choice of open affine neighborhood $U$ of the singularity $x_0$. The following lemma makes this statement precise.

\begin{lem}
Let $V \subset U$ be an open affine neighborhood of $x_0$. Then the inclusion $V \hookrightarrow U$ induces an isomorphism
$H^{n-1}_{\cdh}(U,\mathcal{O}) \cong H^{n-1}_{\cdh}(V,\mathcal{O})$.
\end{lem}

\begin{proof}
Take a Nisnevich cover $\{V \longrightarrow U, V' \longrightarrow U\}$, and then cover $V'$ by open affines $V'_i$. Since $V'$ is smooth, so are all of the $V'_i$, and in particular, they have no higher \emph{cdh}-cohomology groups (Theorem \ref{lem:cdh-Zar-cohomology-isomorphism}). A standard \v{C}ech spectral sequence argument then shows that the induced map is an isomorphism.
\end{proof}

\medskip
\noindent Alternatively, this isomorphism can be obtained directly from Proposition \ref{prop:SES describing K_{1-n}(X)}, by seeing that the kernel $H^{n-1}_{\cdh}(U,\mathcal{O})$ of the map $K_{1-n}(X) \longrightarrow KH_{1-n}(X)$ is independent of the choice of open affine neighborhood $U$ containing the isolated singularity $x_0$.

\medskip
\noindent The discussion using the decomposition \eqref{NK_{1-n}(U) description} yielding the short exact sequence \eqref{SES describing K_{1-n}(X)} is a reasonable description of $K_{1-n}(X)$, but \emph{cdh}-cohomology groups are often difficult to compute. It turns out that we can be more explicit in our description of $K_{1-n}(X)$ in the exact sequence \eqref{eq:spectral sequence abutting to N^tK_q} by identifying the term $H_{\cdh}^{n-1}(U,\mathcal{O})$ in terms of known invariants of the singularity $x_0$.

\begin{defn}
Let $R$ be a finite type $k$-algebra such that $U = \Spec R$ has only isolated singularities. The \emph{generalized Du Bois invariants} $b^{p,q}$ for $p \ge 0, q \ge 1$ are

\begin{equation} \label{defn:Du Bois invariants}
b^{p,q} = \length H^q_{\cdh}(U,\Omega^p).
\end{equation}
\end{defn}

\medskip
\noindent These invariants are finite by \cite{CHWW10b}. Du Bois invariants were introduced by Steenbrink \cite{Ste97}. By \cite[Lemma 2.1, Equation 2.7]{CHWW10b}, we see that $H^{n-1}_{\cdh}(U,\mathcal{O})$ is a $k$-vector space of dimension $b^{0,n-1}$. In particular, its dimension is finite.

\medskip
\noindent Finally, in the case of $n = 3$, we have a full computation of $K_{-2}(X)$.

\begin{cor} \label{cor:full computation of K_{-2}(X)}
Let $X$ be an integral threefold with an isolated singularity $x_0$. Then for any open affine $U$ containing $x_0$, $K_{-2}(X)$ is an extension of $KH_{-2}(X)$ by $H_{\cdh}^2(U,\mathcal{O})$, where $KH_{-2}(X)$ has the description given by Theorem \ref{thm:main theorem for KH_{1-n}(X)}, and $H_{\cdh}^2(U,\mathcal{O})$ is a $k$-vector space of finite dimension $b^{0,2}$.
\end{cor}